\documentclass{amsart}

\usepackage{amsmath}
\usepackage{amsthm}
\usepackage{amsopn}
\usepackage{amssymb}
\usepackage{mathabx}
\usepackage[all]{xy}

\allowdisplaybreaks[1]
\newcommand{\ul}[1]{\underline{#1}}
\newcommand{\mc}[1]{\mathcal{#1}}

\newcommand{\mr}[1]{\mathrm{#1}}
\newcommand{\mbf}[1]{\mathbf{#1}}

\newcommand{\abs}[1]{\left\lvert #1 \right\rvert}

\newcommand{\td}[1]{\widetilde{#1}}

\newcommand{\ZZ}{\mathbb{Z}}

\newcommand{\FF}{\mathbb{F}}

\newcommand{\HH}{\mbf{H}}

\newcommand{\Sp}{\mathrm{Sp}}
\newcommand{\Top}{\mathrm{Top}}

 \newtheorem{thm}[equation]{Theorem}
 
 \newtheorem{lem}[equation]{Lemma}
 \newtheorem{prop}[equation]{Proposition}

\theoremstyle{definition}

 \newtheorem{rmk}[equation]{Remark}
\newtheorem*{thm*}{Theorem}
\newtheorem*{cor*}{Corollary}
\newtheorem*{lem*}{Lemma}
\newtheorem*{prop*}{Proposition}
\newtheorem*{defn*}{Definition}
\newtheorem*{ex*}{Example}
\newtheorem*{exs*}{Examples}
\newtheorem*{rmk*}{Remark}
\newtheorem*{claim*}{Claim}

\numberwithin{equation}{section}
\numberwithin{figure}{section}

\DeclareMathOperator{\Map}{Map}

\DeclareMathOperator*{\hocolim}{hocolim}

\title{A $C_2$-equivariant analog of Mahowald's Thom spectrum theorem}
\author{Mark Behrens and Dylan Wilson}
\date{\today}

\begin{document}

\begin{abstract}
We prove that the $C_2$-equivariant Eilenberg-MacLane spectrum associated with the constant Mackey functor $\ul{\FF}_2$ is equivalent to a Thom spectrum over $\Omega^\rho S^{\rho + 1}$.
\end{abstract}

\maketitle

\section{Introduction}

Let $\mu$ be the M\"obius bundle over $S^1$, regarded as a virtual bundle of dimension $0$.  The mod $2$ Moore spectrum is the Thom spectrum
$$ M(2) \simeq (S^1)^\mu. $$
The classifying map for $\mu$ extends to a double loop map
$$ \td{\mu} : \Omega^2S^3 \rightarrow BO. $$
Mahowald proved the following theorem \cite{Mahowald}:

\begin{thm}[Mahowald]
There is an equivalence of spectra
$$ (\Omega^2S^3)^{\td{\mu}} \simeq H\FF_2. $$
\end{thm}

The bundle $\mu$ may also be regarded as a $C_2$-equivariant virtual bundle over $S^1$, by endowing both $S^1$ and the bundle with the trivial action.  Since $B_{C_2}O$ is an equivariant infinite loop space \cite{Atiyah}, the classifying map for $\mu$ extends to an
$\Omega^{\rho}$-map
$$ \td{\mu}: \Omega^{\rho}S^{\rho+1} \rightarrow B_{C_2}O. $$
Here, $\rho$ is the regular representation of $C_2$. 
The purpose of this paper is to prove the following.

\begin{thm}\label{thm:Mahowald}
There is an equivalence of $C_2$-spectra
$$ (\Omega^\rho S^{\rho+1})^{\td{\mu}} \simeq H\ul{\FF}_2. $$
\end{thm} 

(Here, $\ul{\FF}_2$ denotes the constant Mackey functor with value $\FF_2$.)

\subsection*{Acknowledgements}

Many tricks in this paper have been independently discovered by Doug Ravenel, and the first author's involvement in this project is an outgrowth of mathematical discussions with Agn\'es Beaudry, Prasit Bhattacharya, Dominic Culver, Doug Ravenel, and Zhouli Xu.  The authors also benefited from valuable input from Mike Hill, and the comments of the referee.  The first author was supported by NSF grant DMS-1611786.

\subsection*{Conventions}
Equivariant objects in this paper either live in $\Top^{C_2}$, the category of $C_2$-spaces, or $\Sp^{C_2}$, the category of genuine $C_2$-spectra.  In both of these categories, the equivalences are those equivariant maps which induce equivalences on both the $C_2$-fixed points spectrum and the underlying spectrum.
We let $\HH$ denote the Eilenberg-Maclane spectrum $H\ul{\FF}_2$, with underlying spectrum $H := H\FF_2$.  We use $\HH_\star$ and $\pi^{C_2}_\star$ to denote $RO(C_2)$-graded homology and homotopy \emph{groups} (i.e. \emph{not} the Mackey functors) of $C_2$-equivariant spaces and spectra, and $H_*$ and $\pi_*$ to denote the ordinary homology and homotopy groups of non-equivariant spaces and spectra. We let $\sigma$ denote the sign representation of $C_2$, and let $\rho = 1+\sigma$ denote the regular representation.  For a representation $V$, $S(V)$ denotes the unit sphere in $V$, and $S^V$ denotes its one point compactification, and $\abs{V}$ denotes its dimension.

\section{Equivariant preliminaries}

\subsection*{Euler class}

Let $a$ denote the Euler class in $\pi^{C_2}_{-\sigma}S$, given geometrically by the inclusion
$$ S^0 \hookrightarrow S^\sigma. $$
There is a cofiber sequence
\begin{equation}\label{eq:cofiber}
{C_2}_+ \rightarrow S^0 \hookrightarrow S^\sigma
\end{equation}
so the cofiber of $a$ is stably given by
\begin{equation}\label{eq:Ca}
Ca \simeq \Sigma^{1-\sigma} {C_2}_+. 
\end{equation}
The equivalence of underlying spectra
\begin{equation}\label{eq:ulSsig}
 (S^1)^e \simeq (S^\sigma)^e
\end{equation}
induces an equivalence of $C_2$-spectra
$$ {C_2}_+ \wedge S^1 \simeq {C_2}_+ \wedge S^\sigma. $$
Therefore, the equivalence (\ref{eq:Ca}) can actually be regarded as giving an equivalence
$$ Ca \simeq {C_2}_+. $$
It follows that $Ca$ is a commutative ring spectrum.
The adjoint of the equivalence (\ref{eq:ulSsig}) gives a $C_2$-equivariant map
$$ {C_2}_+ \wedge S^1 \rightarrow S^\sigma $$
which, by the self-duality of ${C_2}_+$, gives a map
$$ u : S^1 \rightarrow {C_2}_+\wedge S^\sigma \simeq Ca \wedge S^\sigma $$   
which serves as a Thom class for the representation $\sigma$.
For $X \in \Sp^{C_2}$, we have
\begin{align*}
\pi^{C_2}_{k}(X) & \cong \pi_k(X^{C_2}), \\
\pi^{C_2}_{V}(X \wedge Ca) & \cong \pi_{\abs{V}}(X^e).
\end{align*}
Said differently,
\begin{equation}\label{eq:pistarCa}
\pi^{C_2}_{\star} (X \wedge Ca) \cong \pi_* (X^e) [u^\pm].
\end{equation}

\subsection*{Tate square}

We will let
\begin{align*}
X^h & := F({EC_2}_+, X), \\
X^\Phi & := X \wedge \td{EC}_2
\end{align*}
denote the homotopy completion and geometric localization of $X$, respectively.  The fixed points of $X^h$ are the homotopy fixed points of $X$, and the fixed points of $X^\Phi$ are the geometric fixed points of $X$.
$X$ is recovered from these approximations by the pullback (``Tate square'') \cite{GreenleesMay}
$$
\xymatrix{
X \ar[r] \ar[d] & X^\Phi \ar[d] \\
X^h \ar[r] & X^t
}
$$
where the spectrum $X^t$ is the equivariant Tate spectrum
$$ X^t := (X^h)^{\Phi}. $$

Note that a generalization of the argument establishing (\ref{eq:Ca}) yields an equivalence
$$ \Sigma^{k\sigma-1}C(a^k) \simeq S(k\sigma)_+. $$
Taking a colimit, we see that we have
\begin{align*}
\hocolim_k \Sigma^{k\sigma-1}C(a^k) & \simeq {EC_2}_+, \\
\hocolim_k S^{k\sigma} & \simeq \td{EC}_2. 
\end{align*}
It follows that homotopy completion and geometric localization can be reinterpreted as $a$-completion and $a$-localization:
\begin{align*}
 X^h & \simeq X^\wedge_a, \\
 X^\Phi & \simeq X[a^{-1}].
\end{align*}
In this manner, the Tate square is equivalent to the ``$a$-arithmetic square''
$$
\xymatrix{
X \ar[r] \ar[d] & X[a^{-1}] \ar[d] \\
X^\wedge_a \ar[r] & X^\wedge_a[a^{-1}]
}
$$
Using (\ref{eq:pistarCa}), the $a$-Bockstein spectral sequence takes the form
$$ E_1^{*,*} = \pi_*(X^e)[u^{\pm}, a] \Rightarrow \pi^{C_2}_\star(X^h). $$
The $a$-Bockstein spectral sequence can be regarded as an $RO(C_2)$-graded version of the homotopy fixed point spectral sequence (see \cite[Lem.~4.8]{HillMeier}).

\subsection*{The mod $2$ Eilenberg-MacLane spectrum}

We have \cite{HuKriz}
$$ \pi^{C_2}_\star \HH = \FF_2[a,u] \oplus \frac{\FF_2[a,u]}{(a^\infty, u^\infty)}\{\theta\} $$
where
\begin{align*}
\abs{u} & = 1-\sigma, \\
\abs{\theta} & = 2\sigma - 2.
\end{align*}
The $a$-$u$ divisible factor in $\pi_\star\HH$ is best understood from the Tate square, using
\begin{align*}
\pi^{C_2}_\star \HH^h & \cong \FF_2[a, u^{\pm1}], \\
\pi^{C_2}_\star \HH^\Phi & \cong \FF_2[a^{\pm1}, u].
\end{align*}
Actually, the second isomorphism lifts to an equivalence
$$ \HH^{\Phi C_2} \simeq H[a^{-1}u] := \bigvee_{i \ge 0} \Sigma^{i} H $$
so we have
$$ \HH^{\Phi}_\star X \cong H_*(X^{\Phi C_2})[a^{\pm1},u] $$
and, restricting the grading to trivial representations, we get
\begin{equation}\label{eq:HPhi}
 \HH^\Phi_* X \cong H_*(X^{\Phi C_2})[a^{-1}u].
\end{equation}

By applying $\pi_{V}^{C_2}$ to the map
$$ \HH \wedge X \rightarrow \HH \wedge X \wedge Ca $$
we get a homomorphism
\begin{equation}\label{eq:Phie}
\Phi^e: \HH_V(X) \rightarrow H_{\abs{V}}(X^e).
\end{equation}
Taking geometric fixed points of a map
$$ S^V \rightarrow \HH \wedge X $$
gives a map
$$ S^{V^{C_2}} \rightarrow \HH^{\Phi C_2} \wedge X^{\Phi C_2} $$
Using (\ref{eq:HPhi}) and passing to the quotient by the ideal generated by $a^{-1}u$, we get a homomorphism
\begin{equation}\label{eq:PhiC2}
\Phi^{C_2}: \HH_V(X) \rightarrow H_{\abs{V^{C_2}}}(X^{\Phi C_2}).
\end{equation}

\subsection*{A useful lemma}

Our main computational lemma is the following.

\begin{lem}\label{lem:useful}
Suppose that $X \in \Sp^{C_2}$ and suppose that $\{b_i\}$ is a set of elements of $\HH_\star (X)$ such that
\begin{enumerate}
\item $\{ \Phi^e(b_i) \}$ is a basis of $H_*(X^e)$, and
\item $\{ \Phi^{C_2}(b_i) \}$ is a basis of $H_*(X^{\Phi C_2})$.
\end{enumerate}
Then $\HH_\star(X)$ is free over $\HH_\star$, and $\{b_i\}$ is a basis.
\end{lem}

\begin{proof}
The set $\{b_i\}$ corresponds to a map
	\[
	\HH \wedge \bigvee S^{|b_i|} \to \HH \wedge X.
	\]
Assumption (1) implies this map is an equivalence upon applying
$\Phi^e$, while assumption (2) implies this map is an equivalence
upon applying $\Phi^{C_2}$. The result follows.
\end{proof}

\section{Homology of $\rho$-loop spaces}

We spell out some specific algebraic structure carried by the equivariant homology of a $\rho$-loop space.  A more detailed and general study of this algebraic structure can be found in \cite{Hill}.

\subsection*{Products}

Suppose $X = \Omega^\rho Y \in \Top^{C_2}$ is a $\rho$-loop space.  Then $X$ is in particular a $1$-loop space, and is therefore an equivariant $H$-space with product
$$ m: X \times X \rightarrow X. $$
However, the $\sigma$-loop space structure also endows $X$ with a twisted product related to the transfer.  Namely, let 
$$ S^\sigma \rightarrow S^\sigma/S^0 \approx {C_2}_+\wedge S^1 $$
be the pinch map.  This gives rise to a twisted product
$$ \td{m}: N^\times \Omega Y \rightarrow \Omega^\sigma Y $$
where
$$ N^\times Z := \Map(C_2, Z) = \underset{\substack{ \curvearrowbotleftright \\ C_2}}{Z \times Z} $$
is the norm with respect to Cartesian product (i.e. the coinduced space).
In particular, there is a map
\begin{equation}\label{eq:mtilde}
\td{m}: N^\times \Omega^2 Y \rightarrow X.
 \end{equation}
Upon applying fixed points to the map (\ref{eq:mtilde}), we get an additive transfer
\begin{equation}\label{eq:transfer}
t: X^e \rightarrow X^{C_2}.
\end{equation}
In homology, the $H$-space structure give rise to a product
$$ m: \HH_V X \otimes \HH_W X \rightarrow \HH_{V+W} X. $$
Using the equivariant commutative ring spectrum structure of $\HH$ \cite{Ullman}, 
the twisted product $\td{m}$ gives rise to a ``norm map''(see \cite[Thm.~7.2]{BlumbergHill})
$$ n: H_k X^e \rightarrow \HH_{k\rho} X. $$

\subsection*{Dyer-Lashof operations}

$X$ has even more structure:
$X$ is an $E_\rho$-algebra \cite{GuillouMay}.  Specifically, regard 
$S(\rho)$ as a $C_2 \times \Sigma_2$-space where $C_2$ acts on $\rho$ and $\Sigma_2$ acts antipodally.  
Then the $E_\rho$-structure gives a map
$$ S(\rho) \times_{\Sigma_2} X^{\times 2} \rightarrow X. $$
Note that $\HH$ is itself an $E_\rho$-ring spectrum, because it is actually an equivariant commutative ring spectrum, so $\HH \wedge X_+$ is an $E_\rho$-ring in $\HH$-modules.
Given $x \in \HH_V(X)$, represented by a map
$$ x: S^V \rightarrow \HH \wedge X_+, $$
there is an induced composite 
\begin{align*} 
\HH \wedge S(\rho)_+ \wedge_{\Sigma_2} S^{2V} \xrightarrow{1 \wedge 1 \wedge x \wedge x} & \HH \wedge S(\rho)_+ \wedge_{\Sigma_2} (\HH \wedge X_+)^{\wedge 2} \\
 \rightarrow & \HH \wedge \HH \wedge X_+ \\
 \rightarrow & \HH \wedge X_+
 \end{align*}
(where the unlabeled maps come from the $E_\rho$-ring and $\HH$-module structure of $\HH \wedge X_+$).
Applying $\pi_\star^{C_2}$, we get 
a total power operation
$$ \mc{P}(x): \td{\HH}_\star(S(\rho)_+ \wedge_{\Sigma_2} S^{2V}) \rightarrow \HH_\star X. $$
For the purposes of this paper we will be only concerned with the case of $V = k\rho - \sigma$ for $k \in \ZZ$.




\begin{prop}
We have
$$ \td{\HH}_\star \left( S(\rho)_+ \wedge_{\Sigma_2} S^{2(k\rho-\sigma)} \right) \cong \HH_\star \{ e_{2k\rho-\sigma-1}, e_{2k\rho - \sigma} \}. $$
\end{prop}

\begin{proof} Consider the following cofiber sequences:
	\begin{align}
	S^{2(k-1)\rho} \to S(\rho)_+ \wedge_{\Sigma_2} S^{2((k-1)\rho)} \to S^{2(k-1)\rho+\sigma}
	\label{cofiber1}
	\\
	\Sigma S(\rho)_+\wedge_{\Sigma_2}S^{2((k-1)\rho)} \to 
	S(\rho)_+\wedge_{\Sigma_2}S^{2(k\rho - \sigma)} \to
	\Sigma^{2(k\rho -\sigma)}S(\rho)_+\label{cofiber2}
	\end{align}
The sequence (\ref{cofiber1}) arises from Theorem 2.15 of \cite{Wilson}
and the second arises from the $(C_2\times \Sigma_2)$-equivariant
inclusion $\Sigma S^{2((k-1)\rho)} \to S^{2(k\rho -\sigma)}$,
where both $C_2$ and $\Sigma_2$ act trivially on the first suspension
coordinate.

In (\ref{cofiber1}), the boundary map on $\HH_{\star}$ is zero because the group
	\[
	\left[S^{2(k-1)\rho+\sigma}, \Sigma^{2(k-1)\rho+1}\HH\right] = \HH_{-1+\sigma}
	\]
is zero. Thus
	\[
	\td{\HH}_{\star}\left(S(\rho)_+\wedge_{\Sigma_2}S^{2(k-1)\rho}\right)
	\cong \HH_{\star}\{e_{2(k-1)\rho}, e_{2(k-1)\rho+\sigma}\}
	=\HH_{\star}\{e_{2k\rho - 2\sigma -2}, e_{2k\rho-\sigma-2}\}.
	\]
Now we turn to the second cofiber sequence. Notice that $S(\rho)_+$
is $C_2$-equivariantly equivalent to $S^{\sigma}\vee S^0$. From the
previous computation, the boundary is then determined
by elements in the following four groups:
	\begin{align*}
	\partial_1\in\left[S^{2(k\rho -\sigma)}, \Sigma^{2k\rho-2\sigma}\HH\right]
	&=\HH_0 \\
	\partial_2\in\left[S^{2(k\rho -\sigma)+\sigma}, \Sigma^{2k\rho-\sigma}\HH\right]
	&=\HH_{0}\\
	\partial_3\in\left[S^{2(k\rho -\sigma)}, \Sigma^{2k\rho-\sigma}\HH\right]
	=\HH_{-\sigma} &= 0\\
	\partial_4\in\left[S^{2(k\rho -\sigma)+\sigma}, \Sigma^{2k\rho-2\sigma}\HH\right] 
	=\HH_{\sigma} &= 0
	\end{align*}
Elements of $\HH_0$ are determined by their restriction to $H_0$, and
comparison with the underlying homology forces $\partial_1=1$ and $\partial_2 = 0$.
The result follows.
\end{proof}

Thus we get a pair of Dyer-Lashof operations
\begin{align*}
Q^{k\rho}: \HH_{k\rho-\sigma}X \rightarrow \HH_{2k\rho-\sigma}X, \\
Q^{k\rho-1}: \HH_{k\rho-\sigma}X \rightarrow \HH_{2k\rho-\sigma-1}X
\end{align*}
given by the formulas
\begin{align*}
Q^{k\rho}(x) & := \mc{P}(x)(e_{2k\rho-\sigma}), \\
Q^{k\rho-1}(x) & := \mc{P}(x)(e_{2k\rho-\sigma-1}).
\end{align*}

\begin{rmk}
If $X$ is actually an equivariant infinite loop space, then $\HH_\star X$ has an action by equivariant Dyer-Lashof operations \cite{Wilson}, and these operations agree with those defined in that paper.
\end{rmk}

\subsection*{Compatibility with fixed points}

The compatibility of all this structure with the maps $\Phi^e$ and $\Phi^{C_2}$ of (\ref{eq:Phie}) and (\ref{eq:PhiC2}) is summarized as follows.

\begin{description}
\item[Products] Note that $X^e$ is an $E_2$-algebra, and $X^{C_2}$ is an $E_1$-algebra.  The maps $\Phi^e$ and $\Phi^{C_2}$ are algebra homomorphisms. \vspace{10pt}

\item[Norms]  The following diagram commutes:
$$
\xymatrix@C+1em{
& H_kX^e \ar[ld]_{t} \ar[d]_{n} \ar[rd]^{\mr{Fr}} \\
H_k X^{C_2} & \HH_{k\rho} X \ar[l]^-{\Phi^{C_2}} \ar[r]_-{\Phi^e} & H_{2k} X^e
}
$$
Here $t$ is the transfer (\ref{eq:transfer}) and $\mr{Fr}$ is the squaring map (Frobenius).\vspace{10pt}

\item[Dyer-Lashof operations]  
The following diagrams commute, where $\epsilon = 0,1$:
$$
\xymatrix@C+1em{
 \HH_{k\rho-\sigma} X \ar[r]^{\Phi^e} \ar[d]_{Q^{k\rho - \epsilon}} &
H_{2k-1}X^e \ar[d]^{Q^{2k-\epsilon}} \\
 \HH_{2k\rho-\sigma-\epsilon} X \ar[r]_{\Phi^e}  &
H_{4k-1-\epsilon}X^e
}
$$
$$
\xymatrix@C+1em{
\HH_{k\rho-\sigma} X \ar[r]^{\Phi^{C_2}} \ar[d]_{Q^{k\rho}}&
H_kX^{C_2} \ar[d]^{\mr{Fr}}  
 \\
 \HH_{2k\rho-\sigma} X \ar[r]_{\Phi^{C_2}}&
H_{2k} X^{C_2}  
}
$$
\end{description}

\section{Homology of $\Omega^\rho S^{\rho+1}$}

\begin{thm}
There is an \emph{additive} isomorphism (of $\HH_\star$-modules)
$$ \HH_\star \Omega^\rho S^{\rho+1} \cong
\HH_\star \otimes E[t_0, t_1, \ldots] \otimes P[e_1, e_2, \ldots]
$$
with
\begin{align*}
\abs{t_i} & = 2^i\rho - \sigma, \\
\abs{e_i} & = (2^i-1)\rho.
\end{align*}
\end{thm}

\begin{proof}
Note that we have
$$ H_* \Omega^2 S^3 = \FF_2[x_1, x_2, \ldots] $$
with
$$ \abs{x_i} = 2^i-1. $$
Here $x_1$ is the fundamental class $\iota_1$, and 
$$ x_i := Q^{2^i}Q^{2^{i-1}}\cdots Q^{2}x_1. $$
Define $t_0 \in \HH_{1} \Omega^\rho S^{\rho+1}$ to be the fundamental class, and define the other ``generators'' $e_i$ and $t_i$ by
\begin{align*}
e_i & := n(x_i), \\
t_i & := Q^{2^i\rho}Q^{2^{i-1}\rho}\cdots Q^{\rho} t_0.
\end{align*}
Consider the product
$$ t^{\ul{\epsilon}}e^{\ul{k}} := t_0^{\epsilon_0}t_1^{\epsilon_1} \cdots e_1^{k_1} e_2^{k_2} \cdots \in \HH_\star(\Omega^\rho S^{\rho+1}) $$
with $\epsilon_i \in \{0,1\}$ and $k_i \ge 0$.  We compute
$$ \Phi^e(t^{\ul{\epsilon}}e^{\ul{k}}) = x_1^{2k_1+\epsilon_0}x_2^{2k_2+\epsilon_1}\cdots. $$
Mapping out of the cofiber sequence (\ref{eq:cofiber}) gives a fiber sequence
$$ \Omega N^\times \Omega S^{\rho+1} \rightarrow \Omega^\rho S^{\rho+1} \rightarrow \Omega S^{\rho+1} \xrightarrow{\Delta} N^\times \Omega S^{\rho+1}. $$
Upon taking fixed points we get a fiber sequence
$$ \Omega^2 S^{3} \xrightarrow{t} (\Omega^\rho S^{\rho+1})^{C_2} \rightarrow \Omega S^{2} \xrightarrow{\mr{null}} \Omega S^{3} $$
In particular there is an equivalence
$$ (\Omega^\rho S^{\rho+1})^{C_2} \simeq \Omega S^2 \times \Omega^2 S^3. $$
and we have
$$ H_*(\Omega^\rho S^{\rho+1})^{C_2} \cong P[y]\otimes P[t(x_1), t(x_2), \ldots] $$
where $y$ is the image of the fundamental class under the map
$$ S^1 \rightarrow (\Omega^\rho S^{\rho+1})^{C_2}. $$
It follows that
$$ \Phi^{C_2}(t^{\ul{\epsilon}}e^{\ul{k}}) = y^{\epsilon_0 + 2\epsilon_1 + 4\epsilon_2 + \cdots } t(x_1)^{k_1} t(x_2)^{k_2}\cdots.  $$
Thus the set
$$ \{ t^{\ul{\epsilon}}e^{\ul{k}} \} \subset \HH_\star X $$
satisfies the hypotheses of Lemma~\ref{lem:useful}, and the result follows.
\end{proof}

\section{The equivariant Mahowald theorem}

In order to prove Theorem~\ref{thm:Mahowald} we will need
to establish a Thom isomorphism

$$ \HH_\star (\Omega^\rho S^{\rho+1})^{\td{\mu}} \cong \HH_\star \Omega^\rho S^{\rho+1}. $$

We will do so in two steps. Recall that an $E_0$-algebra
is just a spectrum $X$ equipped with a map $S^0 \to X$.
Let $\mathrm{Free}^*_{E_{\rho}}: \mathrm{Alg}_{E_0}(\mathrm{Sp}^{C_2})
\to \mathrm{Alg}_{E_\rho}(\mathrm{Sp}^{C_2})$ denote a homotopical left adjoint
to the forgetful functor. An explicit model for this functor is the homotopy pushout
of $E_{\rho}$-algebras:
	\[
	\xymatrix{
	\mathrm{Free}_{E_\rho}(S^0)\ar[d]\ar[r] & \mathrm{Free}_{E_\rho}(X)\ar[d]\\
	S^0 \ar[r] & \mathrm{Free}^*_{E_{\rho}}(X)
	}
	\]

We will need the following theorem.

\begin{thm}\label{thm:freeThom} Let $f: X \to B_{C_2}O$ classify a virtual bundle of dimension zero
and denote by $\tilde{f}: \Omega^{\rho}\Sigma^{\rho} X \to B_{C_2}O$ the
associated $\Omega^{\rho}$-map. Then there is a canonical
equivalence of $E_{\rho}$-algebras
in $\mathrm{Sp}^{C_2}$
	\[
	\mathrm{Free}^*_{E_{\rho}}(X^f) \cong \left(\Omega^{\rho}\Sigma^{\rho}X\right)^{\tilde{f}}.
	\]
\end{thm}
\begin{proof} Combine the equivariant approximation theorem \cite{GuillouMay, RourkeSanderson}
with Theorem IX.7.1 and Remark X.6.4 of \cite{LMS}.
\end{proof}

\begin{rmk} The non-equivariant version of Theorem~\ref{thm:freeThom} was
first observed by Mark Mahowald, and then proven by Lewis. A nice modern
account in the non-equivariant setting via universal properties can be found in
\cite{OmarToby}.
\end{rmk}

\begin{prop}
There is a Thom isomorphism
$$ \HH_\star (\Omega^\rho S^{\rho+1})^{\td{\mu}} \cong \HH_\star \Omega^\rho S^{\rho+1}. $$
\end{prop}

\begin{proof}
Let $\mathrm{Free}^*_{E_{\rho}, \HH}: \mathrm{Alg}_{E_0}(\mathrm{Mod}_{\HH})
\to \mathrm{Alg}_{E_{\rho}}(\mathrm{Mod}_{\HH})$ denote a homotopical left
adjoint to the forgetful functor. Along with the previous theorem, we will need
two facts:
	\begin{enumerate}
	\item $\HH \wedge (-): \mathrm{Sp}^{C_2} \to \mathrm{Mod}_{\HH}$ is
	symmetric monoidal.
	\item There is a Thom isomorphism $\HH \wedge (S^1)^{\mu} \cong
	\HH \wedge S^1_+$.
	\end{enumerate}
The proposition is now proved by the following string of equivalences: 
	\begin{align*}
	\HH \wedge \left(\Omega^{\rho}\Sigma^{\rho}S^1\right)^{\tilde{\mu}}
	&\cong \HH \wedge \mathrm{Free}^*_{E_{\rho}}\left((S^1)^{\mu}\right)
	&\text{by Theorem~\ref{thm:freeThom}}\\
	&\cong \mathrm{Free}^*_{E_{\rho}, \HH}\left(\HH \wedge (S^1)^{\mu}\right)
	&\text{by (1)}\\
	&\cong \mathrm{Free}^*_{E_{\rho}, \HH}\left(\HH \wedge S^1_+\right)
	&\text{by (2)}\\
	&\cong \HH \wedge \mathrm{Free}^*_{E_\rho}\left( S^1_+\right)
	&\text{by (1)}\\
	&\cong \HH \wedge \Omega^{\rho}\Sigma^{\rho}S^1_+. &
	\end{align*}
\end{proof}

\begin{proof}[Proof of Theorem~\ref{thm:Mahowald}]
The Thom class is represented by a map
$$ (\Omega^\rho S^{\rho+1})^{\td{\mu}} \rightarrow \HH. $$
We wish to show this map is an isomorphism on $\HH_\star$.  The homology of $\HH$ is the $C_2$-equivariant Steenrod algebra, computed in \cite{HuKriz} to be
$$ \HH_\star \HH = \HH_\star[\tau_0, \tau_1, \cdots, \xi_1, \xi_2, \cdots ]/(\tau_i^2 = (u+a\tau_0)\xi_{i+1} + a\tau_{i+1}) $$
with 
\begin{align*}
\abs{\tau_i} & = 2^i\rho - \sigma, \\
\abs{\xi_i} & = (2^i-1)\rho.
\end{align*}
It suffices to show it is surjective, since the two homologies are abstractly isomorphic and of finite type.  Observe that the composite
$$ M(2) \simeq (S^1)^{\mu} \rightarrow (\Omega^\rho S^{\rho+1})^{\td{\mu}} \rightarrow \HH $$
hits $\tau_0$.  Everything is hit then, by \cite[Thm.~5.4]{Wilson}.
\end{proof}


\bibliographystyle{amsalpha}
\nocite{*}
\bibliography{C2mah2}

\end{document}